\documentclass[nolineno,a4paper,english,cleveref, autoref, thm-restate]{socg-lipics-v2021}

\pdfoutput=1 
\hideLIPIcs  

\graphicspath{{figs/}}

\usepackage{graphicx}
\usepackage{wrapfig}
\usepackage{hyperref}
\usepackage{enumerate}
\usepackage{complexity}
\usepackage{todonotes}
\usepackage{paralist}
\usepackage{listings}
\usepackage{ifthen}
\usepackage[shortlabels]{enumitem}

\usepackage[font=small,labelfont=bf]{caption}
\usepackage[font=small,labelfont=normalfont,labelformat=simple]{subcaption}
\renewcommand\thesubfigure{(\alph{subfigure})}

\usepackage{wrapfig}

\newboolean{appendix} 
\setboolean{appendix}{true}

\DeclareMathOperator{\rot}{rot}
\DeclareMathOperator{\lex}{lex}

\newcommand{\sigmaast}{\sigma^\ast\!}
\newcommand{\delete}[1]{{\downarrow_{#1}}}
\newcommand{\rotation}[1]{{#1}_{\rot}}
\newcommand{\rotationback}[1]{{#1}_{\rot(-1)}}

\def\calH{{\cal{H}}}
\def\calU{{\cal{U}}}
\def\calD{{\cal{D}}}
\def\calX{{\cal{X}}}
\def\AA{{\cal{A}}}
\def\SS{{\cal{S}}}

\title{An extension theorem for signotopes
} 

\author{Helena Bergold}
{Department of Theoretical Computer Science, Freie Universit\"at Berlin, Germany}
{helena.bergold@fu-berlin.de}
{https://orcid.org/0000-0002-9622-8936}
{DFG Research Training Group 'Facets of Complexity' (DFG-GRK~2434)}
\author{Stefan Felsner}
{Institut f\"ur Mathematik, Technische Universit\"at Berlin, Germany}
{felsner@math-tu-berlin.de}
{https://orcid.org/0000-0002-6150-1998}
{DFG Grant FE~340/13-1}
\author{Manfred Scheucher}
{Institut f\"ur Mathematik, Technische Universit\"at Berlin, Germany}
{scheucher@math-tu-berlin.de}
{https://orcid.org/0000-0002-1657-9796}
{DFG Grant SCHE~2214/1-1}

\authorrunning{H. Bergold, S. Felsner, and M. Scheucher} 

\Copyright{Helena Bergold, Stefan Felsner, and Manfred Scheucher} 

\acknowledgements{We thank the anonymous reviewers for valuable comments.}



\ifthenelse{\boolean{appendix}}{}{
\relatedversion{A full version of the paper is available at \helena{TODO}\url{https://arxiv.org/abs/papernumber}.} }

\supplement{supplemental data and source code is available at~\cite{supplemental_data}.}

\newcommand{\manfred}[1]{\textcolor{blue}{Manfred: #1}}
\newcommand{\helena}[1]{\textcolor{orange}{Helena: #1}}

\numberwithin{theorem}{section} 
\numberwithin{lemma}{section}
\numberwithin{proposition}{section}
\numberwithin{corollary}{section}
\numberwithin{conjecture}{section}
\newtheorem*{definition*}{Definition}
\renewcommand{\paragraph}[1]{\subsubsection*{#1}}

\begin{document}
	
	\maketitle
	
	\begin{abstract}
		In 1926, Levi showed that, for every pseudoline arrangement~$\mathcal{A}$ and two points in the plane, $\mathcal{A}$ can be extended by a pseudoline which contains the two prescribed points. 
	Later extendability was studied for arrangements of pseudohyperplanes in higher dimensions. 
	While the extendability of an arrangement of proper hyperplanes in $\mathbb{R}^d$ with a hyperplane containing  $d$ prescribed points is trivial,
    Richter-Gebert 
	found an arrangement of pseudoplanes in $\mathbb{R}^3$ which cannot be extended with a pseudoplane containing two particular prescribed points. 
		
	In this article, we investigate the extendability of signotopes, 
	which are a combinatorial structure encoding a  rich subclass of pseudohyperplane arrangements. 
    Our main result is that signotopes of odd rank are extendable in the sense that for two prescribed crossing points we can add an element containing them.
	Moreover, we conjecture that in all even ranks $r \geq 4$ there exist signotopes which are not extendable for two prescribed points. 
	Our conjecture is supported by examples in ranks $4$,~$6$,~$8$,~$10$, and $12$
	that were found with a SAT based approach.

\end{abstract}

\section{Introduction}

Given a family of hyperplanes $\mathcal{H}$ in $\mathbb{R}^d$, 
any $d$ points in $\mathbb{R}^d$,
not all on a common hyperplane of $\mathcal{H}$,
define a hyperplane which is distinct from the hyperplanes in~$\mathcal{H}$.
For dimension $d=2$, 
Levi \cite{Levi1926} proved in his pioneering article on pseudoline arrangements 
that the fundamental extendability of line arrangements also applies to
the more general setting of pseudoline arrangements.
A \emph{pseudoline} is a Jordan
curve in the Euclidean plane 
such that its removal from the plane results in two unbounded components,
and 
a \emph{pseudoline arrangement} is a family of pseudolines such that each pair of pseudolines intersects in exactly one point, where the two curves cross properly.

\begin{theorem}[Levi's extension lemma for pseudoline arrangements \cite{Levi1926}]
\label{theorem:levi}
Given an arrangement $\mathcal{A}$ of pseudolines and two points in $\mathbb{R}^2$, not on a common pseudoline of~$\mathcal{A}$. 
Then~$\mathcal{A}$ can be extended by an additional pseudoline which passes through the two prescribed points. 
\end{theorem}

Several proofs for Levi's extension lemma are known today (besides \cite{Levi1926}, see also \cite{ArroyoMRS18,FelsnerWeil2001,Schaefer2019})
and generalizations to higher dimensions have been studied in the context of oriented matroids, which by the representation theorem of Folkman and Lawrence \cite{FolkmanLawrence78} have representations as projective pseudohyperplane arrangements.
For more about oriented matroids, see~\cite{BjoenerLVWSZ1993}.

Goodman and Pollack \cite{GoodmanPollack81} presented an arrangement of 8 pseudoplanes in~$\mathbb{R}^3$ and a selection of three points such that there is no extension of the arrangement with a pseudoplane containing the points.
Richter-Gebert \cite{RichterGebert1993} then investigated a weaker version with only two prescribed points in dimension 3
such that the extending pseudohyperplane contains these two points. He found an example of a rank $4$ oriented matroid on~$8$ elements such that there is no one element extension with an element containing the two prescribed cocircuits. 
With the representation theorem this implies that even the weaker extendability with two prescribed points does not hold. 
The existence of an extension theorem or of counterexamples in higher dimensions/ranks remains open.

\medskip

We present a proof of Levi's extension lemma in a purely combinatorial setting and show that the proof can be adapted to work for higher dimensions. We represent the geometry by 
$r$-signotopes and prove extendability in even dimensions~$d$, that is, when the rank $r=d+1$ is odd; see Theorem~\ref{theorem:levi_odd}.
Surprisingly,
there are non-extendable examples in rank $4$, $6$, $8$, $10$, and $12$.
We conjecture that there 
is no extension theorem 
for any even rank~$r \ge 4$; 
see Conjecture~\ref{conj:levi_even}. 

Signotopes are in close relation to higher Bruhat orders which were introduced by Manin and Schechtman~\cite{manin1989arrangements} and further studied in \cite{Ziegler1993}.
In rank~3, signotopes correspond to pseudoline arrangements in the plane~\cite{FelsnerWeil2001}. In higher ranks they are a subclass of pseudohyperplane arrangements.

Before we formulate our extension theorem for $r$-signotopes,
we introduce some notation and discuss the relation between pseudoline arrangements and 3-signotopes 
(in Section~\ref{ssec:signotopes}). This leads to a reformulation of Levi's extension lemma 
which will be investigated in the context of signotopes of odd rank in Section~\ref{sec:reformulation}.

\subsection{Signotopes}
\label{ssec:signotopes}

Signotopes are a combinatorial structure generalizing permutations and \emph{simple} pseudoline arrangements (i.e., no three pseudolines cross in a common point).
For $r \geq 1$ a \emph{signotope} of \emph{rank} $r$ (short: $r$-signotope) on $n$ elements is a mapping $\sigma$ from $r$-element subsets (\emph{$r$-subsets}) 
of $[n]$ to $+$ or $-$, i.e., 
$ \sigma : \binom{[n]}{r} \to \{+,-\}$
such that 
for every $(r+1)$-subset $X = \{x_1, \ldots, x_{r+1}\}$ (\emph{$r$-packet}) of $[n]$ with $x_1 < x_2 < \ldots< x_{r+1}$ 
there is at most one sign change in the  sequence  
\[
\sigma(X \backslash \{x_{1}\}), \sigma(X \backslash \{x_{2}\}), \ldots, \sigma(X \backslash \{x_{r+1}\}).
\] 

Note that this sequence lists the signs of all induced $r$-subsets of $X$ in reverse lexicographic order.
For $3$-signotopes, the following $8$ sign patterns on $4$-subsets are allowed:
\begin{align*}
++++, \ +++-, \ ++--, \ +---, \
----, \ ---+, \ --++, \ -+++.
\end{align*}

For sake of readability, we write $X = (x_1, \ldots, x_{t})$
to denote a $t$-subset of $[n]$ with sorted elements $x_1 < x_2 < \ldots< x_{t}$. For such an $X$ we denote by $X_j = (x_1,\ldots,x_{j-1},x_{j+1},\ldots,x_{t})$ the set without $x_j$. 
With the convention $- < +$, the condition about sign changes in $r$-signotopes
can be written as a monotonicity condition for $r$-packets $X = (x_1, \ldots, x_{r+1})$:
\[
\sigma(X_1)  \leq \sigma(X_2) \leq \ldots\leq 
\sigma(X_{r+1}) \quad\text{or}\quad \sigma(X_1)  \geq \sigma(X_2) \geq \ldots \geq \sigma(X_{r+1}).
\]

It is well-known 
that every arrangement of pseudolines is isomorphic 
to an arrangement of $x$-monotone pseudolines \cite{Goodman80}.
In such a representation, we label the pseudolines from top to bottom on the left by $1,\ldots,n$. 
Since two pseudolines cross exactly once, the pseudolines appear in reversed order on the right. 
Now the 
corresponding $3$-signotope $\sigma$ is obtained as follows: The sign of $\sigma(a,b,c)$ for $a<b<c$ indicates the orientation of the triangle formed by the pseudolines $a,b,c$ (see Figure~\ref{fig:orientationtriangle}). 
If the crossing of $a$ and $c$ is below $b$, it is $\sigma(a,b,c) = +$ and 
if  the crossing of $a$ and $c$ is above $b$, it is $\sigma(a,b,c) = - $.
In the following we identify the crossings with the elements which cross, i.e. for 3-signotopes crossings are subsets of size~2. 
The 3-signotope $\sigma$ gives information about the partial order of the crossings from left to right. 
If $\sigma(a,b,c) = +$ it holds $ab \prec ac \prec bc$ and if $\sigma(a,b,c) = -$ it is $bc\prec ac\prec ab$.

\begin{figure}[htb]
	\begin{subfigure}[t]{.48\textwidth}
		\centering
		\includegraphics[page = 1]{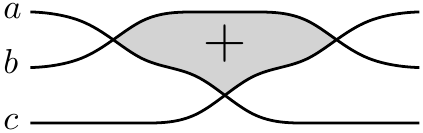} 
		\label{subfig:triangle+}
	\end{subfigure} 
	\begin{subfigure}[t]{.48\textwidth}
		\centering
		\includegraphics[page = 2]{trianglesn3}
		\label{subfig:triangle-}
	\end{subfigure}
	\caption{Connection between pseudoline arrangements and $3$-signotopes.}
	\label{fig:orientationtriangle} 
\end{figure}

Felsner and Weil \cite{FelsnerWeil2001} showed that rank 3 signotopes are in bijection with simple pseudoline arrangements in $\mathbb{R}^2$ with a fixed top cell.
For $r \ge 4$, $r$-signotopes correspond to special pseudohyperplane arrangements in~$\mathbb{R}^{r-1}$, i.e., they are a subclass of oriented matroids of rank~$r$.
A geometric representation of $r$-signotopes in the plane 
is presented in~\cite{Miyata2021} 
(see also~\cite{BalkoFulekKyncl2015} for the rank $3$ case).

\subsection{An extension theorem for signotopes}
\label{sec:reformulation}

In Levi's extension lemma for pseudoline arrangements, 
each of the two prescribed points can either lie
in a cell of the arrangement,
on one pseudoline, 
or be the crossing point of two pseudolines.
To formulate an extension lemma 
in terms of 3-signotopes
we restrict our considerations to simple pseudoline arrangements and to crossing points as prescribed points.
Since the extending pseudoline passes through the two prescribed crossing points, the extension yields a non-simple arrangement. However, by perturbing the extending pseudoline at the non-simple crossing points, we arrive at a simple arrangement, see Figure~\ref{fig:perturbation}.

\begin{figure}[htb]
	\centering
	\includegraphics{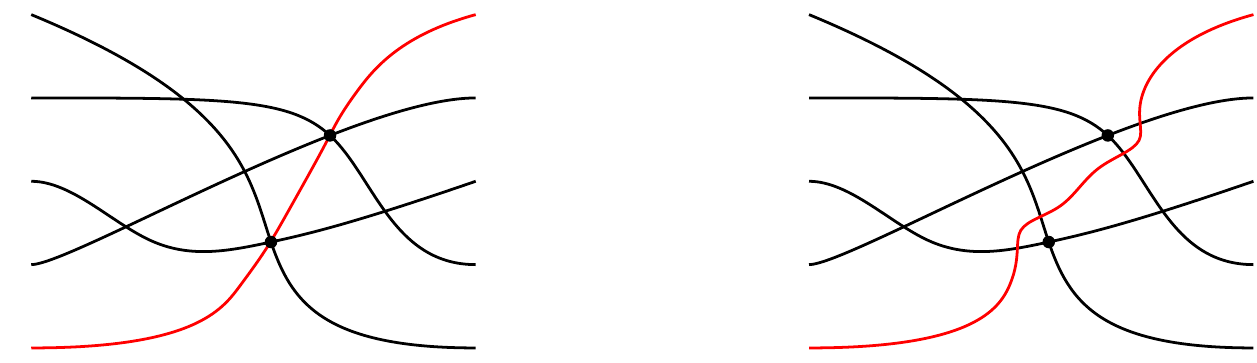}
	\caption{Perturbing an extending pseudoline at the two non-simple crossing points.}
	\label{fig:perturbation} 
\end{figure}

A perturbation 
at a prescribed crossing together with the new inserted pseudoline
yields a \emph{triangular cell} 
incident to the crossing. 
This cell is bounded by
the two pseudolines defining the crossing and the extending pseudoline. 
Triangular cells play an important role in the study of pseudoline arrangements, since it is possible to change the orientation of a triangle by moving one of the pseudolines 
over the crossing of the two others. Such a local perturbation is called a \emph{triangle flip}, it does not change the orientation of any other triangle in the arrangement. 
The triangular cells of the arrangement represented by a $3$-signotope~$\sigma$ are in one to one correspondence with  $3$-subsets such that if we change the sign of this 3-subset in $\sigma$ we obtain a new signotope $\sigma'$.
We call such a $3$-subset a \emph{fliple}. 
The notion of fliples generalizes to higher ranks. In an $r$-signotope $\sigma$ on $[n]$, an $r$-subset $X \subseteq [n]$ is a \emph{fliple} if both assignments $+$ and $-$ to $\sigma(X)$ result in a signotope.
It is worth noting that fliples in signotopes are the analogon of mutations in oriented matroids.
While every signotope contains at least two fliples \cite{FelsnerWeil2001},
it remains a central open problem in combinatorial geometry to decide whether every uniform oriented matroid contains a mutation \cite[Chapter 7.3]{BjoenerLVWSZ1993}.

Let $\AA$ be an arrangement of pseudolines, which are 
labeled $1,\ldots,n$ from top to bottom on the left.
When applying Levi's extension lemma to 
extend $\AA$ the left endpoint of the extending line $\ell$ will be between two consecutive endpoints of pseudolines of $\AA$. To re-establish the properties of the labeling, we have to set the label of $\ell$ accordingly and increase the label of every pseudoline
that starts below $\ell$ by one.
To cope with this relabeling-issue in terms of signotopes,
we introduce the following notion.
For $k \in [n]$ and a subset $X$ of $[n]$,
we define 
\[
X\delete{k} 
= 
\{ x \mid x \in X, x < k\} 
\cup 
\{ x-1 \mid x \in X, x > k\}.
\]
Note that the cardinality of $X$ and $X\delete{k}$ is the same if and only if $k \notin X$.
For an $r$-signotope $\sigma$ on~$[n]$,
we define the \emph{deletion} of an element $k \in [n]$ as  $\sigma\delete{k}$ by 
\[\sigma\delete{k}(X\delete{k}) \mathop{:=} \sigma(X) \]
for all $r$-sets $X \subseteq [n]$ with $k \notin X$. 
This is an $r$-signotope on~$[n-1]$ because each $r$-packet has been an $r$-packet for $\sigma$.

\begin{definition*}
An $r$-signotope $\sigma$ on $[n]$  is \emph{$t$-extendable} if for all pairwise disjoint $(r-1)$-subsets $I_1, \ldots, I_t \in \binom{[n]}{r-1}$, there exists $k \in [n+1]$ and an $r$-signotope $\sigmaast$  on $[n+1]$ 
with fliples $I_1^\ast,\ldots, I_t^\ast$ 
such that $\sigmaast\delete{k} = \sigma$, 
and
$I_j^\ast\delete{k} = I_j$ for all $j = 1, \ldots, t$.
Hence the element $k$ \emph{extends} $\sigma$ to~$\sigmaast$.
\end{definition*}

Note that a $t$-extendable $r$-signotope on $n \ge (r-1)t $ elements is clearly $(t-1)$-extendable.
While the 1-extendability is a simple exercise\footnote{For the sake of completeness, we give a proof of 1-extendability in Corollary~\ref{cor:1extendability} which uses more evolved techniques.}, 
the first interesting part is the $2$-extendability, which we investigate in this paper. 

\begin{theorem}[An extension theorem for signotopes of odd rank]
\label{theorem:levi_odd}
	For every odd rank $r \ge 3$, 
	every $r$-signotope is 2-extendable.
\end{theorem}

Surprisingly, our proof of Theorem~\ref{theorem:levi_odd} 
(see Section~\ref{sec:extension_theorem_odd})
generalizes to the more general setting, 
where the $(r-1)$-subsets $I$ and $J$, 
which are fliples in the extension, intersect.

\begin{corollary}
\label{corollary:nonemptyintersection}
    Let $\sigma$ be an $r$-signotope on $[n]$, $I$ and $J$ be two $(r-1)$-subsets of $[n]$ such that  $|I \cap J| +r$ is odd.
    Then $\sigma$ is extendable to an $r$-signotope $\sigmaast$ on $[n+1]$
    with fliples $I^\ast,J^\ast$ and an \emph{extending} element $k \in [n+1]$
    such that
    $\sigmaast\delete{k} = \sigma$, and
    $I^\ast\delete{k} = I$,
    and
    $J^\ast\delete{k} = J$.
\end{corollary}

Despite the restrictions to simple arrangements and crossing points as prescribed points we can derive Levi's extension lemma (Theorem~\ref{theorem:levi}) in its full generality with little extra work from Theorem~\ref{theorem:levi_odd}.
Details are deferred to \cref{sec:no_restriction}.

The statement of Theorem~\ref{theorem:levi_odd}
applies only to signotopes of odd rank. This is not just a defect of our proof because signotopes in even rank indeed behave differently. For ranks $r = 4,6,8,10,12$ 
we found signotopes on $n = 2r$  elements,
which are not 2-extendable. 
The examples and the source code to verify their correctness are available as supplemental data~\cite{supplemental_data}; see \cref{sec:computerassisted} for details.
Based on these examples, we dare to conjecture:

\begin{conjecture}[No extension theorem for signotopes of even rank]
	\label{conj:levi_even}
	For every even rank $r \ge 4$,
	there~is an $r$-signotope
	which is not $2$-extendable.
\end{conjecture}

\subsection{Signotopes as a rich subclass of oriented matroids}
\label{sec:many_signotopes}

It is well known that the number of oriented matroids of rank $r$ on $n$ elements is $2^{\Theta(n^{r-1})}$ \cite[Corollary~7.4.3]{BjoenerLVWSZ1993}. 
As shown by Balko~\cite[Theorem~3]{Balko19}, $r$-signotopes are a rich subclass of oriented matroids of rank~$r$; 
see 
\ifthenelse{\boolean{appendix}} 
{Appendix~\ref{app:proof_nosignotopes}}
{the full version~\cite{fullversion}} 
for a shorter proof of the following proposition.

\begin{restatable}[{Balko~\cite{Balko19}}]{proposition}{propmanysignotopes}
	\label{prop:many_signotopes}
	For $r \geq 3$,
the number of $r$-signotopes on $[n]$ 
	is $2^{\Theta(n^{r-1})}$.
\end{restatable}

In ranks 1 and 2 there are $2^n$ and $n!$ signotopes on $[n]$, respectively. 
Rank 1 signotopes are mappings from $[n]$ to $\{+,-\}$ without any additional property and 2-signotopes are permutations. 
For rank~$r \ge 3$,
the precise number of $r$-signotopes on~$[n]$ 
has been computed for small values of $r$ and $n$; see 
\href{https://oeis.org/A006245}{A6245} (rank~$3$)
and
\href{https://oeis.org/A060595}{A60595} to \href{https://oeis.org/A060601}{A60601} (rank~$4$ to rank~$10$)
on the OEIS~\cite{OEIS}.

\section{Preliminaries}
\label{sec:prelim}

We now prepare for the proof of Theorem~\ref{theorem:levi_odd}.
As discussed in Section~\ref{ssec:signotopes}, signotopes of rank~3 can be represented by an arrangement with $x$-monotone pseudolines.
The order of the crossings from left to right gives a partial order on the 2-subsets.
In general, $r$-signotopes can be represented by a sweepable arrangement of pseudohyperplanes in $\mathbb{R}^{r-1}$, which similarly allows to define a partial order on $(r-1)$-subsets which correspond to the crossings of $r-1$ elements.
This partial order is combinatorially defined as follows. 
For an $r$-signotope $\sigma$ and every $r$-subset $X = (x_1, \ldots, x_r)$ define:
\begin{align*}
	X_1 \succ X_2 \succ \cdots \succ X_r
	\quad  \text{if}   \quad  
	&\sigma(x_1, \ldots, x_r) = +,
	\qquad \text{and} 
	\\
	X_1 \prec X_2 \prec \cdots\prec X_r
	\quad  \text{if}  \quad 
	&\sigma(x_1, \ldots, x_r) = - .
\end{align*}
Recall that we use the convention $x_1 \leq \ldots \leq x_r$ and $X_i = X \backslash \{x_i\}$.
By taking the transitive closure of all relations obtained from $r$-subsets,
we obtain a partial order on the $(r-1)$-subsets  corresponding to~$\sigma$
 \cite[Lemma~10]{FelsnerWeil2001}.

If we rotate an arrangement of pseudolines, i.e., we choose another unbounded cell as the top cell, we get a pseudoline arrangement 
with the same cell structure.
However, the signotope does not stay the same. 
If we rotate only a single pseudoline, then the 
orientation of the triangle spanned by 3 pseudolines stays the same if and only if the rotated pseudoline is not involved (see for example the triangle spanned by $\{2,3,4\}$ in the left-hand side arrangement, resp.\ $\{1,2,3\}$ in the right-hand side arrangement in Figure~\ref{fig:rotation}). 
When rotating \emph{clockwise}, the first element of $\sigma$ becomes the last one in the rotated signotope $\sigma_{\rot}$.
In terms of the $3$-signotope $\sigma$ the signs of the rotated signotope $\sigma_{\rot}$ are:
$\sigma_{\rot}(a,b,c) = \sigma(a+1,b+1,c+1)$ if $c \neq n$ and 
$\sigma_{\rot}(a,b,n) = - \sigma(1,a+1,b+1)$.

\begin{figure}[htb]
	\includegraphics[page = 1]{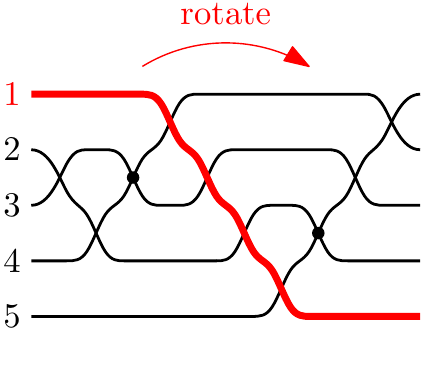} \hfill
	\includegraphics[page = 2]{rotation_newer} \hfill
	\includegraphics[page = 4]{rotation_newer} 
	
	\caption{An illustration of a clockwise rotation of pseudolines. The rotated pseudoline is highlighted red.}
	\label{fig:rotation} 
\end{figure}

In general, 
we define the \emph{clockwise rotated} signotope $\sigma_{\rot}$ of a given $r$-signotope $\sigma$ as:
\begin{align*}
	\sigma_{\rot}(x_1, \ldots, x_r) = 
	\begin{cases}
		-\sigma(1,x_1+1, \ldots, x_{r-1}+1) & \text{ if } x_1 < x_2 < \cdots < x_r=n, \\
		\phantom{-}\sigma(x_1+1, \ldots, x_r+1) & \text{ if } x_1 < x_2 < \cdots < x_r < n.
	\end{cases}
\end{align*}

Here we use the usual convention $- \cdot + = -$ and $- \cdot - = +$.
To keep track of the index shift caused by a clockwise rotation,
we define 
$x_{\rot} = x-1$ if $x \neq 1$ and
$1_{\rot} = n$,
and
\[
X_{\rot} = \{x_{\rot} : x \in X \} = \begin{cases}
	(x_1 -1, x_2 -1, \ldots, x_k-1) \quad & \text{if }  x_1 > 1;\\
	(x_2 -1, \ldots, x_k-1 , n) \quad & \text{if } x_1 = 1
\end{cases} 
\]
for any subset $X = (x_1, \ldots, x_k)$ of $[n]$ with $x_1 < \ldots <x_k$.
Note that this allows us to write $\sigma_{\rot}(X_{\rot}) = \sigma(X)$ if $1\not\in X$ and $\sigma_{\rot}(X_{\rot}) = - \sigma(X)$ if $1\in X$.

As the following lemmas show, this is indeed an $r$-signotope, which moreover has essentially the same fliples. 
The proofs 
and further properties are deferred to \ifthenelse{\boolean{appendix}} 
{Appendix~\ref{sec:properties}.}
{the full version~\cite{fullversion}.} 

\begin{restatable}{lemma}{rotissignotope}
\label{lem:rot_is_signotope}
	Let $\sigma$ be an $r$-signotope on $[n]$.
	Then $\sigma_{\rot}$ is an $r$-signotope on $[n]$.
\end{restatable}

\begin{restatable}{lemma}{rotfliples}
\label{lemma:fliple_rotation}
	Let $\sigma$ be an $r$-signotope and let $F$ be a fliple of $\sigma$. 
	Then $F_{\rot}$ is a fliple in the clockwise rotated signotope~$\sigma_{\rot}$. 
\end{restatable}

\section{Extension theorem for signotopes} 
\label{sec:extension_theorem_odd}

In this section, we give a proof for the extension theorem for signotopes of odd rank. 
The central ingredient of our proof is as follows. 
If  $\sigma$ is an $r$-signotope on $[n]$ and the prescribed two $(r-1)$-sets $I$ and $J$ are incomparable in the partial order associated with $\sigma$ (see Section~\ref{sec:prelim}), then $\sigma$ is extendable by a ``last'' element $n+1$ such that $I \cup \{n+1\}$ and $J \cup \{n+1\}$ are fliples in the extension.
Figure~\ref{fig:perturbation} gives an illustration for the rank~3 case.
More abstractly we can extend the signotope when there is a down-set in the partial order on $(r-1)$-sets which has 
$I$ and $J$ as maximal elements.
A \emph{down-set} of a partial order $(\cal{P}, \prec)$ is a subset $\cal{D} \subseteq \cal{P}$ 
such that for all $p \in \cal{P}$ and $d \in \cal{D}$ with $p \preceq d$ it holds $p \in \cal{D}$.
Similarly, an \emph{up-set} is a subset $\calU \subseteq \cal{P}$ such that for all $p \in \cal{P}$ and $u \in \calU$ with $p \succeq u$ it holds $p \in \calU$.

\begin{proposition}[Extension for incomparable elements] \label{prop:down-set}
	Let $(\cal{P},\prec)$ be the partial order on $(r-1)$-sets corresponding to an $r$-signotope~$\sigma$ on $[n]$. 
	For every down-set $\cal{D} \subseteq \cal{P}$ there exists an $r$-signotope $\sigmaast$ on $[n+1]$ such that all $r$-subsets
	of the form $M \cup \{n+1\}$, where $M$ is a maximal element of $\cal{D}$, are fliples of $\sigmaast$ and $\sigmaast\delete{n+1} = \sigma$.
\end{proposition} 

\begin{proof}
	Define the extended $r$-signotope $\sigmaast$ on $[n+1]$ as follows:
	\begin{align*}
		\sigmaast(x_1 \ldots, x_r) = \begin{cases}
			\sigma(x_1, \ldots, x_r) & \text{ if } x_1, \ldots, x_r \in [n];\\
			+ & \text{ if } x_r = n+1 \text{ and } \{x_1, \ldots, x_{r-1}\} \in \cal{D}; \\
			- & \text{ if } x_r = n+1 \text{ and } \{x_1, \ldots, x_{r-1}\} \not\in \cal{D}.
		\end{cases}
	\end{align*}

    First we show that $\sigmaast$ is an $r$-signotope on $[n+1]$.
    Consider an $r$-packet $X = (x_1, \ldots x_{r+1})$. 
    We need to show that the sequence 
    \begin{align*}
    	\sigmaast(X_{1}), \sigmaast(X_{2}), \ldots, \sigmaast(X_{r+1})
    \end{align*}
    has at most one sign change. 
    
    If $x_{r+1}\leq n$, then all signs on the considered $r$-subsets are the same as for $\sigma$. Since $\sigma$ is an $r$-signotope, there is at most one sign change in the sequence. 
    
    In the other case, we have $x_{r+1} = n+1$.
    For all $j \leq r$ we have $n+1 \in X_j$.
    Furthermore, $\sigmaast (X_{r+1}) = \sigma(X_{r+1})$ because $n+1 \not\in X_{r+1}$.
    We consider two cases. First, if $\sigma(X_{r+1}) = +$ we have by definition of the partial order 
    \begin{align*}
    	X\backslash \{x_{r+1}, x_i \} \succ X\backslash \{x_{r+1}, x_j \} \qquad \text{ for } i<j.
    \end{align*} 
    By the property of a down-set this means that, whenever $X\backslash \{x_{r+1}, x_i \} \in \calD$, we also have $X\backslash \{x_{r+1}, x_j \} \in \calD$ for $i<j$. 
    Let $i^{\ast}$ be the smallest integer such that $X\backslash \{x_{r+1}, x_{i^{\ast}} \} \in \calD$. Then by definition of $\sigmaast$ we have $\sigmaast(X_j) = -$ for all $j < i^\ast$ and $\sigmaast(X_j) = +$ for all $j \geq i^\ast$.

    Similar arguments apply if $\sigma(X_{r+1}) = -$.
    Then we have 
    \begin{align*}
    	X\backslash \{x_{r+1}, x_i \} \prec  X \backslash \{x_{r+1}, x_j \} \qquad \text{ for } i<j.
    \end{align*} 
    This time let $i^{\ast}$ be the smallest integer such that $X\backslash \{x_{r+1}, x_{i^{\ast}} \} \not\in \cal{D}$. Then by definition of~$\sigmaast$ we have $\sigmaast(X_j) = +$ for all $j \leq i^\ast$ and $\sigmaast(X_j) = -$ for all $j > i^\ast$.
    
    Let $M$ be a maximal element of the down-set $\calD$. By the analysis above it follows that  $M \cup \{n+1\}$ is
    adjacent to a sign change in each packet in which it is contained. Hence it is a fliple. 
\end{proof}

From this proposition it follows that for all $r\geq 2$ all $r$-signotopes are 1-extendable.
Moreover the 1-extension contains the extending element at the last position.
\begin{corollary}[1-extendability]
\label{cor:1extendability}
    For $r \geq 2$ let $\sigma$ be an $r$-signotope on $[n]$ and $I$ an $(r-1)$-subset. 
    Then there is an extending $r$-signotope $\sigmaast$ on $[n+1]$ elements such that $I \cup \{n+1\}$ is a fliple and $\sigmaast\delete{n+1} = \sigma$.
\end{corollary}

The following two propositions show that,
for odd rank,
we can always find a rotation of the corresponding signotope such that the two prescribed $(r-1)$-subsets are incomparable. 
We can then use Proposition~\ref{prop:down-set} to define an extension.

\begin{restatable}{proposition}{propositionrotincomp}
\label{prop:rot_incomp}
    Let $\sigma$ be an $r$-signotope on $[n]$.
    For two $(r-1)$-subsets $I,J$ with $I \prec J$ and $1 \notin I \cap J$, it holds
    $I_{\rot}$ and $J_{\rot}$ are incomparable in $\prec_{\rot}$ 
    or $I_{\rot} \prec_{\rot} J_{\rot}$.
\end{restatable}
The proof of Proposition~\ref{prop:rot_incomp} needs some more structural properties of the partial order and its interaction with the rotation. The details are deferred to
\ifthenelse{\boolean{appendix}}
{\cref{sec:properties}.}
{the full version~\cite{fullversion}.}

\begin{proposition} \label{lem:rotate}
	Let $r\geq 3$ be an odd integer, 
	let $\sigma$ be an $r$-signotope on~$[n]$
	and let $I,J$ be two disjoint $(r-1)$-subsets. 
	After at most $n-1$ clockwise rotations,
	$\sigma$, $I$, and~$J$ are transformed into $\sigma'$, $I'$, and~$J'$, resp., 
	such that $I'$ and~$J'$ are incomparable in the partial order $\prec'$ corresponding to~$\sigma'$.
\end{proposition}

\begin{proof}
	Assume $I$ and $J$ are comparable in the partial order $\prec$ corresponding to the $r$-signotope $\sigma$ 
	with $I \prec J$, otherwise we are done.
	We show that after $n$ clockwise rotations, all signs of $\sigma$ are reversed. 
	Hence the partial order $\prec'$ corresponding to the (possible multiple times) rotated signotope $\sigma'$ is the reversed relation to $\prec$. 
	
	The sign of an $r$-subset $(z_1, \ldots, z_r)$ changes from $+$ to $-$ or vice versa if and only if the rotated element is contained in $(z_1, \ldots, z_r)$, i.e., if we rotate $z_1$.
	Hence after rotating $n$ times in total every $z_i$ was rotated and thus the sign of an $r$-subset changes exactly $r$ times. 
	Since $r$ is odd, the sign after rotating $n$ times is opposite. The obtained signotope $\sigma'$ is the reverse of the original signotope $\sigma$ and the corresponding partial order is also reversed.  
	
    Furthermore we cannot reverse the order of two disjoint $(r-1)$-sets in one rotation as shown in  Proposition~\ref{prop:rot_incomp}. 
	Hence there will be a moment where the two disjoint sets are incomparable.
\end{proof}

Although the following lemma is trivial in the setting of pseudoline arrangements, 
we need to prove it in the context of general $r$-signotopes. 
We show that the extension of a rotated signotope when rotated back does contain the original signotope. To show this we need to investigate the interaction between the rotation and deletion of elements. 

\begin{lemma}\label{lem:rotateback}
    Let $\sigma$ be an $r$-signotope on $[n]$ and $x \in [n]$. Then it is $\sigma_{\rot} \delete{n} = \sigma\delete{1}$ and  $\sigma_{\rot} \delete{x_{\rot}} = (\sigma\delete{x}) _{\rot}$ for $x \neq 1$.
\end{lemma}

\begin{proof}
    Because of the index shift it does not matter whether we delete the first element or we rotate $\sigma$ such that in the first element becomes the last and delete the last element in this rotated signotope. Hence the first part $\sigma_{\rot} \delete{n} = \sigma\delete{1}$ holds.
    
    Now assume $x \neq 1$ which implies $x_{\rot} \neq n$.
    Both mappings are $r$-signotopes on $[n-1]$.
    We need to check whether they map to the same signs.
    Let $X$ be an $r$-subset of $[n-1]$ and let $X^{\ast}$ be an $r$-subset of $[n]$ with $x_{\rot} \notin X^\ast$ and $X^{\ast} \delete{x_{\rot}} = X$. 
    We obtain  
    \begin{align*}
        \sigma_{\rot}\delete{x_{\rot}} (X) =
        \sigma_{\rot}\delete{x_{\rot}} (X^{\ast}\delete{x_{\rot}}) =
        \sigma_{\rot} (X^{\ast}).
    \end{align*}
    We will now rewrite the term to get the statement. 
    Recall that rotating an $r$-signotope on $n$ elements exactly $2n$ times results in the original signotope. Hence rotating $2n-1$ times corresponds to a counterclockwise rotation, i.e., the inverse operation of a clockwise rotation. 
    We denote this counterclockwise rotation by $\rot(-1)$.
    Since $x_{\rot} \notin X^\ast$, we have $x \notin \rotationback{(X^{\ast})}$.
    By definition it is 
    \begin{align*}
        \rotation{\sigma}(X^{\ast}) 
        = \varepsilon \cdot \sigma(\rotationback{(X^{\ast})}) 
        = \varepsilon \cdot \sigma \delete{x} ((\rotationback{(X^{\ast})}) \delete{x})
        =\varepsilon \cdot \sigma \delete{x} (\rotationback{X})
        = \rotation{(\sigma \delete{x})}  (X),
    \end{align*}
    where the sign $\varepsilon = +$ (resp.\ $\varepsilon = -$) if $n \notin X^\ast$ (resp.\ $n  \in X^\ast$).
    Note that $n \in X^\ast$ is equivalent to $1 \in \rotationback{X}$ for $x \neq 1$.
    This completes the proof of the lemma. 
\end{proof}

With
Proposition~\ref{prop:down-set}, Proposition~\ref{lem:rotate} and Lemma~\ref{lem:rotateback} we are now ready to prove Theorem~\ref{theorem:levi_odd}.

\subsection{Proof of Theorem~\ref{theorem:levi_odd}}
\label{thm:2extendableoddrank}

    Let $\sigma$ be an $r$-signotope on $[n]$ 
    and let $I,J$ be a pair of disjoint $(r-1)$-subsets.
    By Proposition~\ref{lem:rotate} there exists $k \in \{0,\ldots, n-1\}$
    such that 
    the $k$-fold rotated $(r-1)$-subsets $I_{\rot(k)}$, $J_{\rot(k)}$ are incomparable 
    in the $k$-fold rotated signotope $\sigma_{\rot(k)}$.
    
    To extend the signotope $\sigma_{\rot(k)}$, 
    we use the down-set $\calD$ consisting of $I_{\rot(k)}$, $J_{\rot(k)}$,
    and all $(r-1)$-subsets below. 
    In this down-set $I_{\rot(k)}$ and $J_{\rot(k)}$ are maximal elements since they are incomparable. 
    Hence we can apply Proposition~\ref{prop:down-set} in order to add a new element at position $n+1$ in the rotated signotope $\sigma_{\rot(k)}$ such that $I_{\rot(k)} \cup \{n+1\}$ and $J_{\rot(k)} \cup \{n+1\}$ are fliples.
    The extended signotope is denoted by $\sigma^\ast_{\rot(k)}$ and fulfills $\sigma^\ast_{\rot(k)} \delete{n+1} = \sigma_{\rot(k)}$.
    
    Finally, we need to find a rotation of $\sigma^\ast_{\rot(k)}$ which contains the original signotope~$\sigma$. 
    For this we perform $k+1$ counterclockwise rotations 
    (or equivalently, $2n+1-k$ clockwise rotations)
    and denote the so-obtained signotope by $\sigma^\ast$.
    Note that we perform $k+1$ counterclockwise rotations since the newly added element needs to be rotated and the $k$-fold clockwise rotation needs to be undone.
    After $k+1$ counterclockwise rotations, the added element $n+1$ in $\sigma^\ast_{\rot(k)}$ becomes the element $k+1$ in~$\sigma^\ast$. 
    It remains to show that $\sigma^\ast \delete{k+1} = \sigma$.
    
    After the first counterclockwise rotation, 
    the added element $n+1$ in $\sigma^\ast_{\rot(k)}$ becomes the first element~1 in~$(\sigma^\ast_{\rot(k)}){\rotationback{}}$. 
    By Lemma~\ref{lem:rotateback} it holds
    $((\sigma^\ast_{\rot(k)}){\rotationback{}})\delete{1} = 
    (\sigma^\ast_{\rot(k)})\delete{n+1} = 
    \sigma_{\rot(k)} $.
    After additional $k$ counterclockwise rotations, the added element $n+1$ in $\sigma^\ast_{\rot(k)}$ 
    becomes the element $k+1$ in~$\sigma^\ast$. 
    Furthermore $I \cup \{k+1 \} $ and $ J \cup \{k+1 \} $ are fliples of $\sigma^\ast$ by Lemma~\ref{lemma:fliple_rotation}. 
    Since we do not rotate the extending element, applying the second part of  Lemma~\ref{lem:rotateback} multiple times shows 
    $((\sigma^\ast_{\rot(k)})_{\rot(-1)})\delete{1} = (\sigma^\ast\delete{k+1})_{\rot(k)} $.
    Together with the previous equation
    this shows $\sigma_{\rot(k)} = (\sigma^\ast\delete{k+1})_{\rot(k)}$, which further implies $\sigma = \sigma^\ast\delete{k+1}$.
    Hence we obtain the signotope $\sigma$ when deleting $k+1$ from $\sigma^\ast$.
    This completes the proof of Theorem~\ref{theorem:levi_odd}.

\subsection{Proof of Corollary~\ref{corollary:nonemptyintersection}}
\label{sec:nonemptyintersection}

To prove Corollary~\ref{corollary:nonemptyintersection}, 
we proceed similar as in the proof of Theorem~\ref{theorem:levi_odd}.
By Proposition~\ref{prop:down-set}, it suffices to show that after some rotations the  $(r-1)$-subsets corresponding to $I$ and $J$ are incomparable.

    Let $s = |I \cap J |$. Since Theorem~\ref{theorem:levi_odd} covers the case $s=0$, we may assume $s \ge 1$.
    We consider the following two cases. 
    
    First, assume that $r$ is odd and $s$ is even.
    For odd rank $r$, we have already seen that after $n$ rotations, the signotope is reversed and hence the corresponding partial order is reversed. 
    For even $s$, the relation between $I$ and $J$ is reversed $s$ times (whenever we rotate one element $x \in I \cap J$). 
    These are the only $s$ times where we reverse the order in one single rotation. 
    Since $s$ is even and the order is reversed after $n$ rotations, the corresponding $(r-1)$-subsets must be incomparable in between.
    
    Second, assume that $r$ is even and $s$ is odd.
    For even rank $r$, the $n$-fold rotation leaves the signotope unchanged
    and hence also the partial orders are the same. 
    Since $s$ is odd, we reverse the orientation of $I$ and $J$ exactly $s$ times in a single rotation step. 
    Hence they must be incomparable in between. 
    
The statement now follows from Proposition~\ref{prop:down-set} and Lemma~\ref{lem:rotateback} similar as in the proof of Theorem~\ref{thm:2extendableoddrank}.
This completes the proof of Corollary~\ref{corollary:nonemptyintersection}.

\section{Examples in even rank: SAT attack and properties}
\label{sec:computerassisted}

Since the proof for the extension theorem (Theorem~\ref{theorem:levi_odd})
applies only for odd ranks,
we had to investigate even ranks in a different manner.
For rank~4, we used computer assistance to enumerate
all signotopes and then tested
2-extendability for each signotope.
On 6 and 7 elements all 4-signotopes are 2-extendable. 
On 8 elements we found non-2-extendable 4-signotopes. 
For both, the enumeration and the 2-extendability test,
we modeled SAT instances which were then solved using 
the python interfaces \verb|pycosat| \cite{pycosat}  and  \verb|pysat| \cite{pysat} to run the
SAT solver \verb|picosat|, version 965, \cite{Biere08} and \verb|cadical|, version~1.0.3 \cite{Biere2019}, respectively.

Using this two-level-SAT approach we managed to find
the first examples of 4-signotopes which are not 2-extendable.
In order to keep symmetries and similarities of our nicely structured example of rank $4$, we restricted our search space to examples in rank~$r$ on $2r$ elements. 
While for rank 4 all signotopes on 8 elements can be enumerated within a few seconds,
the complete enumeration 
in higher ranks is unpractical
as the number of $r$-signotopes on $2r$ elements grows faster than doubly exponential in~$r$ (cf.\ Proposition~\ref{prop:many_signotopes}). 
Hence, to be able to approach higher ranks, 
we further analyzed the structure of our non-2-extendable rank~4 examples together with an analyze of the already found rank~6 examples. 
These made it possible to find a recursive construction.
See Section~\ref{ssec:structure} for more details.

With the observed properties as additional constraints,
we further restricted the search space
so that only ``reasonable'' candidates were enumerated. 
Under these restrictions, we managed to find examples for rank 6, 8, 10, and 12 which are not 2-extendable.

\subsection{SAT model for enumeration}
\label{ssec:satenum}

To encode $r$-signotope on $n$ elements,
we proceed as following.
We use Boolean variables $S_{X}$ for every $X \in \binom{[n]}{r}$
to indicate whether $\sigma(X)=+$.
To ensure that these variables model a valid signotope,
we add constraints which ensure that for every $r$-packet $Y = \{y_1,\ldots,y_{r+1}\} \in \binom{[n]}{r+1}$
there is at most one sign-change in the sequence 
\[
\sigma(Y_{1}),
\ldots,
\sigma(Y_{r+1}).
\]

More precisely, 
since there are exactly $2r+2$ possible assignment of this sequence,
we introduce
auxiliary variables $T_{Y,t}$ for $ t \in \{1,\ldots,2r+2\}$
to indicate which of the assignments applies.\footnote{
Alternatively one can assert 
$
\neg S_{Y_i} \vee
S_{Y_j} \vee
\neg S_{Y_k}
$
and
$
S_{Y_i} \vee
\neg S_{Y_j} \vee
S_{Y_k}
$ 
for every $Y$ and $ 1 \le i<j<k \le r+1$.
Even though this approach does not require auxiliary variables to indicate the types of $(r+1)$-tuples, we need these auxiliary variables to assign the variables for fliples later anyhow.
}

Next 
we introduce auxiliary variables
$F_{X,Y}$ for every $r$-packet $Y \in \binom{[n]}{r+1}$ and every $r$-tuple $X \in \binom{Y}{r}$
to indicate whether $X$ is a fliple when $\sigma$ is restricted to~$Y$. This is done in a similar fashion as for the $S_X$ variables.
Using the $F_{X,Y}$ variables, 
we can assert the variables 
$
F_X = \bigvee_{Y \in \binom{[n]}{r+1} \colon X \subset Y} F_{X,Y}
$
for every $X \in \binom{[n]}{r}$
to indicate whether $X$ forms a fliple.
Last but not least, 
we introduce variables
$L_{X,k}$
to indicate whether $X$ is the $k$-th fliple. 
This will allow us to enumerate only  configurations with a prescribed number of fliples.

\subsection{SAT model for testing 2-extendability}
\label{ssec:satext}

We are now ready to 
formulate a SAT instance to decide whether
a given signotope $\sigma$ on $[n]$ 
and given disjoint $(r-1)$-tuples $I,J$
can be extended by an additional element $n+1$
such that $I \cup \{n+1\}$ and $J \cup \{n+1\}$
are fliples in the extension~$\sigma^*$.
This is sufficient to test extendability since whenever there is an extension, there is a rotation such that the signotope is extendable by an element at the last position. 
As in Section~\ref{ssec:satenum},
we create a SAT instance 
to find an $(n+1)$-element signotope
but we add constraints to fix $\sigma$
and to assert that  $I \cup \{n+1\}$ and $J \cup \{n+1\}$ are fliples in $\sigma^*$.

For a given signotope $\sigma$ on elements $[n]$ 
we can now iterate over all disjoint $(r-1)$-tuples $I,J$
and test whether there is a rotation
of $\sigma$ and $I,J$ 
such that in the extension $\sigma^*$ by the element $n+1$
the $r$-tuples $I \cup \{n+1\}$ and $J \cup \{n+1\}$ are fliples.
If for some $I,J$ no such rotations exists,
we have certified that $\sigma$ is not 2-extendable.

\subsection{Structure of the examples supporting Conjecture \ref{conj:levi_even}}
\label{ssec:structure}

In order to find the first witnessing examples for \cref{conj:levi_even} in rank $4$, we used the two-step SAT approach as described in Sections~\ref{ssec:satenum} and \ref{ssec:satext}.
To make investigations in higher ranks,
we had to get a better understanding of the examples found in rank~4.
Hence we filtered those with regularities and symmetries to come up with a generalization of the observed properties and analyzed their structure.
Our aim was to find a relation between examples in different ranks, for example using projection and deletion arguments.
For this we investigated the structure of our rank 4 examples together with some already found rank 6 examples.

One of the first and crucial observations was that
there exist signotopes such that 
for every choice of even indices $I \subset E_r \mathop{:=} \{2,4,\ldots, 2r\}$ and every choice of odd indices $J \subset O_r \mathop{:=} \{1,3,\ldots, 2r-1\}$ there is no such extension.
In fact, for such examples
it is sufficient to check $I=\{2,4,\ldots, 2r-2\}$ and $J=\{1,3,\ldots, 2r-3\}$ to verify the non-2-extendability.
This observation not only allowed us to restrict the search space, but also to speed up the extendability-test by a factor of $\Theta(r^2)$ 
since not all pairs of $(r-1)$-tuples $I,J$ need to be tested.

While we came up with further observations one by one over the time, 
we here give a summary of all the properties,
which we desire from the examples in rank $r$ with $n=2r$ elements.
In the following we denote by $X = (x_1, x_2, \ldots, x_r)$ an $r$-tuple and use the notation
$(-)^i = +$ if $i$ is even and $(-)^i = -$ if $i$ is odd.

\begin{enumerate}[(a),  labelsep = 1em, leftmargin = 4em]
    \item $\sigma = \sigma_{\rot(4)}$, where $\sigma_{\rot(4)}$ is obtained by the 4-fold rotation of~$\sigma$.
    
    \item $\sigma(2,4,\ldots, 2r) = -$ and $\sigma(1, 3, \ldots, 2r-1) = +$.
    
    \item If there is only one even or only one odd element in~$X$, then the sign $\sigma(X)$ depends only on the position of that element in~$X$.
    More specifically: If $e = x_i$ is the only even element in~$X$, then 
    $
        \sigma(X) = (-)^i. 
    $
    If $o=x_i$ is the only odd element in~$X$, then it is
    $
        \sigma(X) = (-)^{i+1}. 
    $
    
    \item If $x_1, \ldots, x_i \in E_r$ and $x_{i+1}, \ldots, x_r \in O_r$ with $2 \leq i \leq r-2$, then the sign is 
    $
        \sigma(X) = (-)^{i+1}. 
    $
    
    \item 
    Let $x_1, \ldots, x_i \in O_r$ and $x_{i+1}, \ldots, x_r \in E_r$ for $2 \leq i < r-2$.
    If $x_r < 2r$,
    then $\sigma(X) = -.$
    If $x_{j} = 2j$ for all $j = i+1, \ldots, r$, 
    then $\sigma(X) = +.$    

\end{enumerate}

\noindent
Furthermore, we fix the following set of 8 fliples for rank 4. 
\begin{align*}
    F_4 = \{ &(1, 3, 5, 7), (2, 4, 6, 8 ), (2, 3, 7, 8), (1, 3, 4, 8), \\
    &(1, 2, 4, 7),(3, 5, 6, 8), ( 4, 5, 7, 8),
(3, 4, 6, 7)
\}
\end{align*}
Together with the 4-fold symmetry it is sufficient to mention only some of them:
\begin{align*}
    \widehat{F}_4 = \{(1, 3, 5, 7), (2, 4, 6, 8), (4, 5, 7, 8), (3, 4, 6, 7), (1, 2, 4, 7)\}
\end{align*}

\noindent
In rank 4, there are only four signs which are not determined by the above properties:
    \begin{align*}
        (1,3,4,8), \qquad (4,5,7,8), \qquad (2,3,7,8),
        \qquad (3,4,6,7)
    \end{align*}
    
\noindent
By the 4-fold symmetry,
the assignment of $(1,3,4,8)$ also determines the sign of $(4,5,7,8)$, and vice versa. 
The third and fourth tuple have a similar interaction.
Hence, there are precisely 4 signotopes in rank~4 which fulfill the above properties.
We fix one of the four configurations 
(the choice does not play a role)
and refer to it as $\sigma_4$ in the following.

\medskip
In order to find examples in higher ranks, we use the following property.
\begin{enumerate}[(a), labelsep = 1em, leftmargin = 4em]
\setcounter{enumi}{5}
    \item Let $\sigma_{r-2}$ be an example of rank $r-2$ on $2r-4$ elements. 
    For an $r$-tuple $X \subseteq [2r]$
    with $1,3 \notin X$ and $2,4 \in X$,
    we define the sign 
    \begin{align*}
        \sigma_r ( X) = \sigma_{r-2}(X\delete{1,2,3,4}),
    \end{align*}
    where $X\delete{1,2,3,4} = (((X\delete{4})\delete{3})\delete{2})\delete{1}$ denotes the $(r-2)$ tuple on $[2r-4]$.
    Note that $X\delete{1,2,3,4}$ is obtained by deleting the elements $2$ and $4$ from~$X$ and a further index shift by $-2$  caused by deleting\footnote{Inspired by the language of oriented matroids, such an operation might be called \emph{contraction}.} $1$ and~$3$, which are not contained in $X$. 
\end{enumerate}

Altogether,
if we start 
with one example from rank~4 
and recursively construct examples in higher ranks   
with the desired properties and further prescribe $(r/2)^2+(r/2)+2$ fliples for rank~$r$, 
it finally turned out that there is a unique example in each of the ranks $r=6,8,10,12$. 
All examples and the source code to verify their correctness are available as supplemental data~\cite{supplemental_data}.

In the future we hope   
to find an argument for the non-2-extendability based on the described properties
and construct an infinite family of examples.
Even though we conjecture that there is an infinite family,
we want to clarify that we found examples in rank~4 and~6 which do not have the above properties and hence the assumptions might also be too strong.

\section{Theorem~\ref{theorem:levi_odd} implies Levi's extension lemma (Theorem~\ref{theorem:levi})}
\label{sec:no_restriction}

\begin{figure}[tb]

	\centering
	\includegraphics{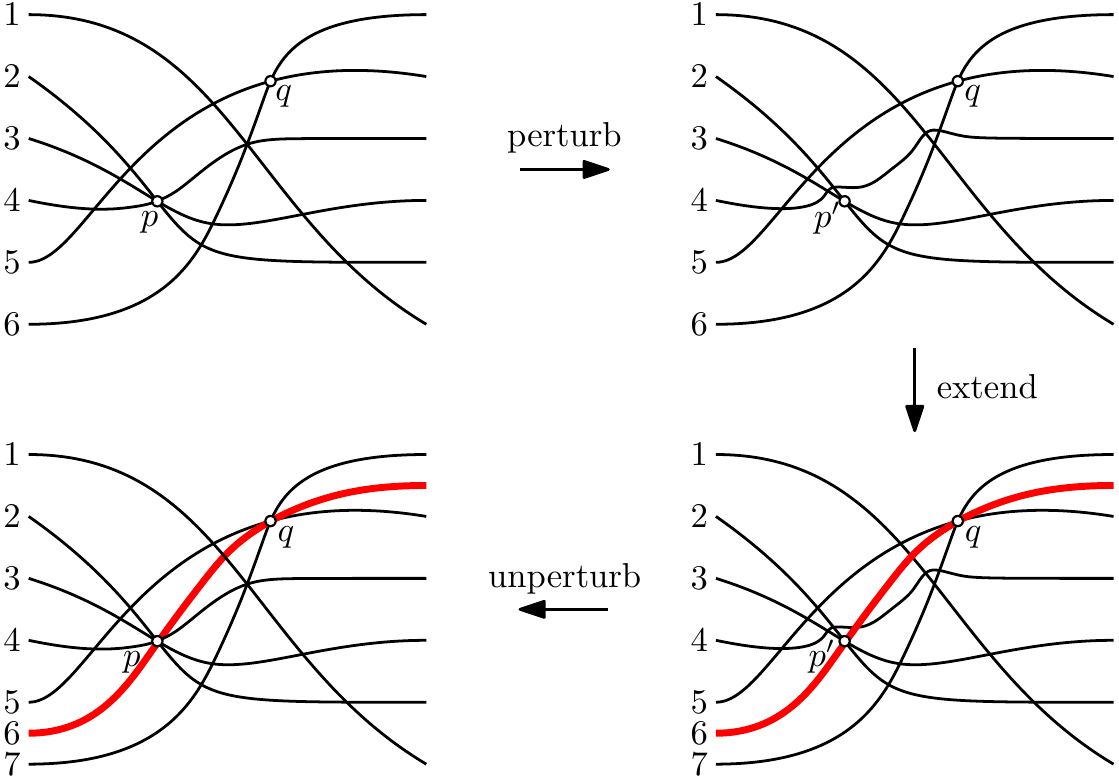}
		
	\caption{An illustration how Theorem~\ref{theorem:levi_odd} implies Levi's extension lemma (Theorem~\ref{theorem:levi}).
	When perturbing the top-left arrangement, 
	the multi-crossing point $p$ (the intersection of $2$, $3$, and $4$)
	is split into three simple crossing points, 
	including the point $p'$ (the intersection of $2$ and $3$).
	After the extension,
	we again contract these three crossing points to one multi-crossing point.
	}
	\label{fig:levi_via_signotopes} 
\end{figure}

As outlined in Section~\ref{sec:reformulation},
it is sufficient to prove 
Levi's extension lemma for 
simple arrangements of pseudolines 
and for crossing points as prescribed points.
Given a non-simple arrangement, 
we can slightly perturb the multiple crossing points 
(as depicted in Figure~\ref{fig:perturbation})
to obtain a simple arrangement. We obtain simplicial cells instead of the multiple crossings. 
This simple arrangement can then be extended, 
and each of the multiple crossing points of the original arrangement 
can again be obtained by contracting the simplicial cells to a point.
Whenever a prescribed point lies on a pseudosegment  
or inside a cell,
we can extend the arrangement through an adjacent crossing.
By perturbing the extending pseudoline, 
we can ensure that the pseudoline passes through the originally prescribed point.

{
	\small
	\bibliography{references}
}

\ifthenelse{\not\boolean{appendix}} 
{}
{\newpage
\appendix 

\section{Asymptotic number of signotopes}
\label{app:proof_nosignotopes}

In this section, we give a short proof for Proposition~\ref{prop:many_signotopes}.

\propmanysignotopes*

\subsection{Proof of the upper bound}

For the upper bound we could just use the fact, 
that $r$-signotopes on $n$ elements are rank~$r$ oriented matroids
and their number is upper bounded by $2^{c(n^{r-1})}$ \cite[Chapter~7.4]{BjoenerLVWSZ1993}. For completeness, however, we include the inductive proof.

For rank~$3$, 
there exists a constant $c>0$ 
such that for every $n$ 
there are at most $2^{cn^2(1+o(1))}$ signotopes on $n$ elements. 
The currently best bound $c=0.657$ is provided in \cite{FelsnerValtr2011}.

For rank~$r\ge4$, we proceed by induction.
Given an $r$-signotope $\sigma$ on $[n]$, 
we compute its projections.
For each $i \in [n]$, we project $\sigma$ to $i$ and obtain an $(r-1)$-signotope $\sigma/_{i}$ on $n-1$ elements. Formally, $\sigma/_{i}$ is defined by 
$\sigma/_{i}(J\delete{i}) \mathop{:=} \sigma(J)$ for every $r$-subset $J$ with $i \in J$.
Since two distinct $r$-signotopes yield different sequences $(\sigma/_{i})_{i \in [n]}$ of projections, 
we can bound the number of $r$-signotopes as 
\[
f_r(n) \le (f_{r-1}(n-1))^n \le \left(2^{c(n-1)^{r-2}} \right)^n \le 2^{cn^{r-1}}  ,
\]
where $f_{r-1}(n-1)$ denotes the number of $(r-1)$-signotopes on $n-1$ elements.

\subsection{Proof of the lower bound}

For convenience we assume  $n = rm$ for some $m \in \mathbb{N}$.
We partition $[n] = \bigcup_{k=1}^r N_k$ with intervals $N_k = [(k-1)m+1,km]$ of size $m$.

On $r$-subets we define the weight $\phi(x_1,\ldots,x_r) = \sum_{k=1}^{r-1} x_k  - x_r$.
Note that,
for $r$-packets $X=(x_1,\ldots,x_{r+1})$ 
with $x_1 < \ldots < x_{r+1}$ as usual,
it holds 
$\phi(X_1) > \ldots > \phi(X_{r})$ and
$\phi(X_r) < \phi(X_{r+1})$.

For a threshold $T$ we now define a collection $\SS_T$ of signotopes on $[n]$. 
A signotope $\sigma$ is in~$\SS_T$ if $\sigma$ has signs as follows, where $\pm$ indicates that the sign can be chosen arbitrarily.
\[
 \sigma(x_1,\ldots,x_r) 
  = \begin{cases}
	- & \text{if $x_{r-1} \not \in N_r$, $x_r \in N_r$, and } \phi(x_1,\ldots,x_r) > T \\
	\pm & \text{if $x_{r-1} \not \in N_r$, $x_r \in N_r$, and } \phi(x_1,\ldots,x_r) = T \\
	+ & \text{otherwise}.
\end{cases}  
\]

For the lower bound we show two properties:
\begin{itemize}
\item The elements of  $\SS_T$ are indeed signotopes.
\item For fixed $r$ and a suitable chosen $T$ there are $2^{\Omega(n^{r-1})}$ elements in $\SS_T$.
\end{itemize}

Let us now check the monotonicity of all $r$-packets $X=(x_1,\ldots,x_{r+1})$,
that is, there is at most one sign-change in the sequence 
\[
\sigma(X_1),\ldots,\sigma(X_{r+1}).
\]
If $x_{r+1} \not \in N_r$, then $\sigma(X_k) = +$ for all $k=1,\ldots,r+1$ 
and there is no sign change on the packet. 
Otherwise, there is some $k \in [r+1]$ such that $x_1,\ldots,x_{k-1} \not \in N_r$ and $x_k,\ldots,x_{r+1} \in N_r$.

If $k<r$, then $x_{r-1},x_r,x_{r+1} \in N_r$. Hence each $X_j$ contains at least two elements from~$N_r$ whence
$\sigma(X_j) = +$ for all $j$ and there is no sign change on the packet.

If $k=r$, then each  $X_j$ with $j < r$ contains two elements from $N_r$ and thus $\sigma(X_j)=+$ for $j < r$.
We also know that $\phi(X_r) < \phi(X_{r+1})$, therefore, if $\sigma(X_r)=-$ then $\sigma(X_{r+1})=-$ as well.
Hence, there is at most one sign change on the packet.

Finally, if $k=r+1$, then $\sigma(X_{r+1})=+$ and $\phi(X_1) > \phi(X_2) > \ldots  > \phi(X_{r})$
whence $\sigma(X_1),\sigma(X_2), \ldots,\sigma(X_{r})$ is 
a sequence of $-$ signs followed by a sequence of $+$ signs possibly one $\pm$ in between.
Again there is at most one sign change on the packet.

This completes the proof that all elements of $\SS_T$ are signotopes.

It remains to show that for some $T$ the set $\SS_T$ contains sufficiently many elements.

Call an $r$-tuple $(x_1,\ldots,x_r)$ \emph{splitted} if $x_k \in N_k$ for
$k=1,\ldots,r$. Splitted $r$-tuples with $\phi(x_1,\ldots,x_r) = T$ are
tuples where elements of $\SS_k$ can freely and independently choose the sign from $+$ and $-$.
Hence, if the number of these tuples is $a_T$
then $\displaystyle |\SS_T| \geq 2^{a_T}$.

Consider the equation
\begin{equation}
	\label{eqn:sum_equal}
	\sum_{k=1}^{r-1} \left(x_k - (k-1)m \right) = x_r - (r-1)m.
\end{equation}
and note that this is equivalent to
\begin{equation}
	\sum_{k=1}^{r-1} x_k - x_r = m \frac{(r-4)(r-1)}{2} = T.
\end{equation}
If we define $ T= \frac{(r-4)(r-1)}{2}$
and
$y_k = x_k - (k-1)m$, 
then we have $1\leq y_k \leq m$ because
$x_k \in N_k$ and $\sum_{k=1}^{r-1} y_k  = y_r$ by equation~\eqref{eqn:sum_equal}.
Clearly such $y$ vectors and splitted $r$-tuples with $\phi(x_1,\ldots,x_r) = T$ are
in bijection. Now 
vectors $(z_1,\ldots,z_{r-1})$ with $1\leq z_1 <\ldots < z_{r-1} \leq m$
are in bijection with the $y$ vectors via $z_k = \sum_{j \leq k} y_j$.
The number of $z$ vectors is just the number of $(r-1)$-subsets of $[m]$.

Hence, for $T= m \frac{(r-4)(r-1)}{2}$ the number of
splitted $r$-tuples with $\phi(x_1,\ldots,x_r) = T$ is
\[
a_T= \binom{m}{r-1} =
\Theta(\frac{1}{r!}\left(\frac{n-r}{r}\right)^{r-1}) = \Theta(n^{r-1}).
\]
This completes the proof of Proposition~\ref{prop:many_signotopes}.

\section{Properties of the clockwise rotation}
\label{sec:properties}

In this section, we prove some properties of the rotation of a signotope which 
play a central role in the proof of the extension theorem (Theorem~\ref{theorem:levi_odd}).

\rotissignotope*

\begin{proof}
	Consider an $r$-packet 
	$X' = (x_1',\ldots,x_{r+1}')$.
	Since rotation is a bijection on the $r$-packets of $[n]$
	there is an $X = (x_1, \ldots, x_{r+1})$ such that $X' = X_{\rot}$.
	
	If the rotated element is not in~$X$, 
	i.e., $x_{r+1}' < n$, then $x_i = x_i'+1$ for all $i = 1, \ldots r+1$
	and the signs of the $r$-tuples of packet $X'$ have to be considered in the same order:
	\[
	\sigma_{\rot}(X'_1), \sigma_{\rot}(X'_2), 
   \ldots, \sigma_{\rot}(X'_r), \sigma_{\rot}(X'_{r+1}) =
   	\sigma(X_1), \sigma(X_2), 
   \ldots, \sigma(X_{r}), \sigma(X_{r+1})
   \]
    the latter has at most one sign change since $\sigma$ is an $r$-signotope.
	
	If the rotated element is in~$X$, 
	that is, $x_{r+1}'=n$ and $x_1 = 1$, then we have
	$x_{i+1} = x_i' +1$ for all $i = 1, \ldots, r$.
	Note that $n \in X'_i$ for $i = 1, \ldots, r$ and hence $1 \in X_j$  for $j = 2, \ldots, r+1$.
	The ordered sequence of signs given by $X'$ is
	\begin{align*}
		&\sigma_{\rot}(X'_1) &&\sigma_{\rot}(X'_2), 
		&& \ldots, 
		&&\sigma_{\rot}(X'_{r}),
		&&\sigma_{\rot}(X'_{r+1}) \\
		= \ 
		&\sigma_{\rot}(X'\backslash \{x_2-1\}), &&\sigma_{\rot}(X'\backslash \{x_3-1\}), 
		&& \ldots, 
		&&\sigma_{\rot}(X'\backslash \{x_{r+1}-1\}),
		&&\sigma_{\rot}(X' \backslash \{n\}) \\
		= \ 
		&-\sigma(X_2), 
		&&-\sigma(X_3), 
		&&\ldots , 
		&&-\sigma(X_{r+1}),
		&&\sigma(X\backslash\{1\}) \\
		= \ 
		&-\sigma(X_2), 
		&&-\sigma(X_3), 
		&&\ldots , 
		&&-\sigma(X_{r+1}),
		&&\sigma(X_1) 
	\end{align*}
	which has at most one sign change because
	$\sigma(X_1), \sigma(X_2), 
   \ldots, \sigma(X_{r}), \sigma(X_{r+1})$
	has at most one sign change due to the signotope property of $\sigma$.
\end{proof}

The following lemma shows that the rotated signotope $\sigma_{\rot}$ has essentially the same properties as $\sigma$ when it comes to fliples. We need to handle only the index shift.  

\rotfliples*
\begin{proof}
	To prove that an $r$-subset $F_{\rot}$ is a fliple, we need to check all  $r$-packets~$X'$ with~$F_{\rot} \subset X'$ as shown in the previous proof. Let $X'$ be such a packet and let $X$ be such that
	$X_{\rot} = X'$.
	Since $F$ is a fliple in $\sigma$ we know that if we change the sign of $\sigma(F)$ there is still at most one sign change in the sequence $\sigma(X_1), \sigma(X_2), 
   \ldots, \sigma(X_{r}), \sigma(X_{r+1})$, we abreviate this by saying that $F$ is \emph{flipable} in $X$. 
   
   If $1\not\in X$, then $\sigma(X_i) = \sigma_{\rot}(X'_i)$  for all $i$. Moreover if $j$ is such that $F=X_j$ then $F_{\rot} = X'_j$, hence, $F_{\rot}$ is flipable in $X'$.
   
   Otherwise we have $1\in X$. Then as shown in the proof of Lemma~\ref{lem:rot_is_signotope} it is
   \[
	\sigma_{\rot}(X'_1), \sigma_{\rot}(X'_2), 
   \ldots, \sigma_{\rot}(X'_r), \sigma_{\rot}(X'_{r+1}) =
   	-\sigma(X_2), -\sigma(X_3), 
   \ldots, - \sigma(X_{r+1}) , \sigma(X_{1}).
   \]
   We know that $F$ is flipable in $X$. If $F=X_i$ with 
   $3\leq i\leq r$, then $F_{\rot} = X'_{i-1}$ is clearly flipable in $X'$. If $F=X_1$ the sequence $\sigma(X_2), 
   \ldots,\sigma(X_{r+1})$ is constant. This implies that   
   the sign of $\sigma_{\rot}(F_{\rot}) = \sigma_{\rot}(X'_{r+1})$ can be fliped. 
   If $F=X_2$ then $\sigma(X_1) \neq \sigma(X_3)$ 
   and the signs $\sigma_{\rot}(X'_i)$ for $2\leq i \leq r+1$ are the same. Hence $F_{\rot} = X'_1$ is flipable in $X'$.
   If $F=X_{r+1}$ the sequence $\sigma(X_1),\ldots,\sigma(X_{r})$ is constant, whence the sign of $\sigma_{\rot}(F_{\rot}) = \sigma_{\rot}(X'_{r})$ is adjacent to different signs and can thus be fliped.
   
This shows that $F_{\rot}$ is flipable in all packets containing it and hence a fliple. 
\end{proof}

\begin{lemma}\label{lem:rotint}
	Let $\sigma$ be an $r$-signotope with partial order $\prec$  and $\sigma_{\rot}$ the rotated signotope with corresponding partial order $\prec_{\rot}$. 
	For two $(r-1)$-subsets $I,J$ with an intersection $|I \cap J| = r-2$ and $I \prec J$ it holds 
		\begin{align*}
			I_{\rot} \prec_{\rot} J_{\rot} \quad  & \text{ if } 1 \notin I \cap J, \quad \text{and} \\
			I_{\rot} \succ_{\rot} J_{\rot} \quad  & \text{ if } 1 \in I\cap J .
		\end{align*}
\end{lemma}

\begin{proof}
	If $1 \notin I$ and $1 \notin J$, then $1 \notin I \cup J$ and the sign of $ I \cup J$ is the same for $\sigma$ and $\sigma_{\rot}$, i.e. $\sigma(I \cup J) = \sigma_{\rot} ( I _{\rot} \cup J_{\rot})$.
	Furthermore the order of $I$ and $J$ in the $(r-1)$-packet $I \cup J$ is the same as the order of $I_{\rot}$ and $J_{\rot}$ in the $(r-1)$ packet $I_{\rot} \cup J_{\rot}$. 
    Hence if $I$ is lexicographically larger than $J$, then $I_{\rot}$ is lexicographically larger than $J_{\rot}$.
    Hence $I_{\rot} \prec_{\rot} J_{\rot}$.
	
	If $1 \in I$ but $1 \notin J$ the sign of the $I$ is after the sign of $J$ in the sign sequence corresponding to the $(r-1)$-packet $I \cup J$ which corresponds to the reversed lexicographic order. 
	By assumption it is $J \succ I$ and thus $\sigma(I \cup J) = +$.
	After rotating clockwise, the sign of $I_{\rot}$ is before the sign of $J_{\rot}$ in the $(r-1)$-packet $I_{\rot} \cup J_{\rot}$.
	since $n \in I_{\rot}$ and $n \notin J_{\rot}$.
	Furthermore the sign of the $r$-subset changes, i.e., $\sigma(I \cup J) = - \sigma_{\rot}(I_{\rot} \cup J_{\rot}) = -$. 
	This shows the relation stays the same, i.e., $I_{\rot} \prec_{\rot} J_{\rot}$.
	The case $1 \in J$ but $1 \notin I$ works analogously.
	
	If $1 \in I$ and $1 \in J$ the order of the appearance of $I$ and $J$ in the $(r-1)$-packet $I \cup J$ is the same as the order of $I_{\rot}$ and $J_{\rot}$ in $I_{\rot} \cup J_{\rot}$ 
	but the sign of the $r$-tuple gets is reversed, i.e., $\sigma(I \cup J) = - \sigma_{\rot}(I_{\rot} \cup J_{\rot})$. 
	Thus the order between $I_{\rot}$ and $J_{\rot}$ is reversed as claimed. 
\end{proof}

A central role in the proof of Theorem~\ref{theorem:levi_odd} is the relation after rotation for two arbitrary $(r-1)$-subsets.
We show that the order of two disjoint elements cannot be reversed with a single rotations. 

For the proof of Proposition~\ref{prop:rot_incomp},
we introduce the following two partitions.
With respect to the first element~$1$, we partition 
the $(r-1)$-subsets 
$\binom{[n]}{r-1}$ into the following three sets:
\begin{align*}
    \calH^{\sigma}_1 &= \{\  I \subset [n] : |I| = r-1, \ 1 \in I \ \} \\
    \calU^{\sigma}_1 &= \{\  I \subset [n] : |I| = r-1, \ 1 \notin I, \ \sigma(I \cup \{1\}) = + \ \} \\
    \calD^{\sigma}_1 &= \{\  I \subset [n] : |I| = r-1, \ 1 \notin I, \ \sigma(I \cup \{1\}) = - \ \} .
\end{align*}
Similarly,
with respect to the last element~$n$, we partition 
$\binom{[n]}{r-1}$ into the following three sets:
\begin{align*}
    \calH^{\sigma}_n &= \{\  I \subset [n] : |I| = r-1, \ n \in I \ \} \\
    \calU^{\sigma}_n &= \{\  I \subset [n] : |I| = r-1, \ n \notin I, \ \sigma(I \cup \{n\}) = - \ \} \\
    \calD^{\sigma}_n &= \{\  I \subset [n] : |I| = r-1, \ n \notin I, \ \sigma(I \cup \{n\}) = + \ \} .
\end{align*}
Note the sign change in the definition, that is, every $I \in \calU^{\sigma}_1$ fulfills $\sigma(I \cup \{1\})=+$ while every $I \in \calU^{\sigma}_n$ fulfills $\sigma(I \cup \{n\}) = -$.

\begin{lemma}
    \label{lemma:upset_downset}
    $\calU^{\sigma}_1$ and $\calU^{\sigma}_n$ are up-sets 
    and $\calD^{\sigma}_1$ and $\calD^{\sigma}_n$ are down-sets 
    of the partial order $\prec$ corresponding to the $r$-signotope~$\sigma$.
\end{lemma}

\begin{proof}
    In the following we show that $\calU^{\sigma}_1$ is an up-set. Analogous arguments show that $\calU^{\sigma}_n$ is an up-set and 
    that $\calD^{\sigma}_1$ and $\calD^{\sigma}_n$ are down-sets.
    Let $I$ be an element of $\calU^{\sigma}_1$.
    By definition, it is $1 \notin I$ and $\sigma(I \cup \{1\}) = +$.
    Let $J$ be an $(r-1)$-subset with $J \succ I$.

    If the intersection $I \cap J$ contains $r-2$ elements, 
    we cannot have $1 \in J$, 
    as otherwise $J$ was lexicographic smaller than $I$, that is, $J$ appears in the $(r-1)$-packet $I \cup J$ after $I$, and thus $- = \sigma(I \cup J) = \sigma(I \cup \{1\}) = +$, a contradiction.
    Therefore, $1 \notin J$ and we have $(r+1)$ elements in $I \cup J \cup \{1\}$.
    If $I$ is lexicographic smaller than $J$, 
    we have the lexicographical order $I \cup J \succ_{\lex} J \cup \{1\} \succ_{\lex} I \cup \{1\}$ which corresponds to the order in the $r$-packet $I \cup J \cup \{1\}$.
    Since we have $\sigma(I \cup \{1\}) = +$ by assumption and $\sigma(I \cup J) = +$ because $J \succ I$, it follows $\sigma(J \cup \{1\}) = +$. Hence $J \in \calU^{\sigma}_1$.
    
    In the other case, if $J$ is lexicographical smaller than $I$, we have the lexicographical order 
    $I \cup J \succ_{\lex}  I \cup \{1\} \succ_{\lex} J \cup \{1\}$.
    Since we have $\sigma(I \cup \{1\}) = +$ and $\sigma(I \cup J) = -$, 
    it follows $\sigma(J \cup \{1\}) = +$ and hence again $J \in \calU^{\sigma}_1$.
    
    If the intersection $ I \cap J$ contains less than $r-2$ elements, we proceed by induction.
    There is a chain $I = Z_1 \prec Z_2 \prec \cdots \prec Z_k = J$ such that any two consecutive $Z_i$ have an intersection of $r-2$ elements. 
    For $i=2,\ldots,k$, 
    since $Z_{i-1} \in \calU^{\sigma}_1$, 
    we conclude that $Z_i \in \calU^{\sigma}_1$, and in particular, $J \in \calU^{\sigma}_1$.
    This completes the proof that $\calU^{\sigma}_1$ is an up-set.
\end{proof}

We now study the effect of a clockwise rotation to the partial order. 
In the partial order $\prec_{\rot}$ corresponding to the rotated signotope $\sigma_{\rot}$, 
the sets 
$(\calU^{\sigma}_1)_{\rot}
$ 
and 
$(\calD^{\sigma}_1)_{\rot}
$ 
remain up-set and down-set, respectively.
Here $\calX_{\rot} = \{ X_{\rot} \colon X \in \calX \}$ denotes the clockwise rotated sets of a set-system $\calX$. 

\begin{lemma}
    \label{lemma:upset_rotate}
    It holds
    $(\calH^{\sigma}_1)_{\rot} = \calH^{\sigma_{\rot}}_n$,
    $(\calU^{\sigma}_1)_{\rot} = \calU^{\sigma_{\rot}}_n$,
    and
    $(\calD^{\sigma}_1)_{\rot} = \calD^{\sigma_{\rot}}_n$.
\end{lemma}
    
\begin{proof}
	An $(r-1)$-subset $I$ contains the first element~$1$ if and only if its clockwise rotation $I_{\rot}$ contains the last element~$n$.
	Therefore, we have $(\calH^{\sigma}_1)_{\rot} = \calH^{\sigma_{\rot}}_n$ and $(\calU^{\sigma}_1 \cup \calD^{\sigma}_1)_{\rot} = \calU^{\sigma_{\rot}}_n \ \cup \ \calD^{\sigma_{\rot}}_n$.
	To show $(\calU^{\sigma}_1)_{\rot} = \calU^{\sigma_{\rot}}_n$ and $(\calD^{\sigma}_1)_{\rot} = \calD^{\sigma_{\rot}}_n$,
	it suffices to prove $(\calU^{\sigma}_1)_{\rot} \subseteq \calU^{\sigma_{\rot}}_n$ and $(\calD^{\sigma}_1)_{\rot} \subseteq \calD^{\sigma_{\rot}}_n$.
	
    To show $(\calU^{\sigma}_1)_{\rot} \subseteq \calU^{\sigma_{\rot}}_n$, let $I \in \calU^{\sigma}_1$, i.e., 
    $\sigma(I \cup \{1\}) = +$. 
    After rotating the element~$1$, we obtain 
    \begin{align*}
        \sigma_{\rot}((I \cup \{1\})_{\rot})
    = - \sigma(I \cup \{1\}) = -.
    \end{align*}
    Since $(I \cup \{1\})_{\rot} = I_{\rot} \cup \{n\}$, we have $\sigma_{\rot}(I_{\rot} \cup \{n\}) = -$ and thus $I_{\rot} \in \calU^{\sigma_{\rot}}_n$. 
    An analogous argument shows  $(\calD^{\sigma}_1)_{\rot} \subseteq \calD^{\sigma_{\rot}}_n$.
    This completes the proof of Lemma~\ref{lemma:upset_rotate}. 
\end{proof}

With the above lemmas, 
we can now prove Proposition~\ref{prop:rot_incomp}.

\propositionrotincomp*

\begin{proof}
    Assume towards a contradiction that $I,J$ are two $(r-1)$-subsets 
    with $I \prec J$ 
    and $I_{\rot} \succ_{\rot} J_{\rot}$.
    
    If $I \in \calU^{\sigma}_1$,
    then by Lemma~\ref{lemma:upset_downset},
    $J \in \calU^{\sigma}_1$.
    If $I \in \calD^{\sigma}_1$,
    then by Lemma~\ref{lemma:upset_rotate},
    $I_{\rot} \in \calD^{\sigma_{\rot}}_n$ and by Lemma~\ref{lemma:upset_downset} and the assumption that $I_{\rot} \succ_{\rot} J_{\rot}$ it is 
    $J_{\rot} \in \calD^{\sigma_{\rot}}_n$
    and again, by Lemma~\ref{lemma:upset_rotate}
    $J \in \calD^{\sigma}_1$.
    Analogous arguments show that,
    if $J \in \calD^{\sigma}_1$ 
    (resp.~$J \in \calU^{\sigma}_1$), 
    then $I \in \calD^{\sigma}_1$ 
    (resp.~$I \in \calU^{\sigma}_1$).
    
    Since $1 \notin I \cap J$ not both $I$ and $J$ can be in $\calH_1^\sigma$.
    Hence $I$ and $J$ are both in $\calD^{\sigma}_1$ or both in $\calU^{\sigma}_1$.
    Since $I \prec J$, there is a chain $I=Z_1 \prec \ldots \prec Z_k = J$. By Lemma~\ref{lemma:upset_downset} it is  $Z_1,\ldots,Z_k \in \calD^{\sigma}_1$ (resp.~$\calU^{\sigma}_1$).
    After a clockwise rotation, we have $(Z_1)_{\rot},\ldots,(Z_k)_{\rot} \in \calD^{\sigma_{\rot}}_n$ (resp.~$\calU^{\sigma_{\rot}}_n$) and hence
    $I_{\rot} = (Z_1)_{\rot} \prec_{\rot} \ldots \prec_{\rot} (Z_k)_{\rot} = J_{\rot}$,
    which is a contradiction to $I_{\rot}  \succ_{\rot} J_{\rot}$. This completes the proof. 
\end{proof}
It is worth noting that 
for $I,J \in \calH^{\sigma}_1$ (i.e., $1 \in I \cap J$)
with $I \prec J$
Lemma~\ref{lemma:upset_downset} implies that any chain $I=Z_1 \prec \ldots \prec Z_k = J$
lies entirely in~$\calH^{\sigma}_1$ 
(i.e., $Z_1,\ldots,Z_k \in \calH^{\sigma}_1$).
Since a clockwise rotation converts comparability of elements containing the element~1,
we have $I_{\rot}=(Z_1)_{\rot} \succ_{\rot} \ldots \succ_{\rot} (Z_k)_{\rot} = J_{\rot}$.

} 

\end{document}